\documentclass[preprint,12pt]{elsarticle}
\usepackage{amsthm,amsmath,amssymb}
\usepackage{graphicx}
\usepackage{longtable,booktabs}
\usepackage[colorlinks=true,citecolor=black,linkcolor=black,urlcolor=blue]{hyperref}
\usepackage{mathrsfs}
\usepackage{epstopdf}
\usepackage{epsfig}

\theoremstyle{plain}
\newtheorem{theorem}{Theorem}
\newtheorem{lemma}[theorem]{Lemma}
\newtheorem{corollary}[theorem]{Corollary}
\newtheorem{proposition}[theorem]{Proposition}

\theoremstyle{definition}
\newtheorem{definition}[theorem]{Definition}
\newtheorem{example}[theorem]{Example}
\newtheorem{conjecture}[theorem]{Conjecture}

\theoremstyle{remark}
\newtheorem{remark}[theorem]{Remark}

\begin{document}
\title{Partial-dual genus polynomials and signed intersection graphs}
\author{Qi Yan\\
\small School of Mathematics\\[-0.8ex]
\small China University of Mining and Technology\\[-0.8ex]
\small P. R. China\\
Xian'an Jin\footnote{Corresponding author.}\\
\small School of Mathematical Sciences\\[-0.8ex]
\small Xiamen University\\[-0.8ex]
\small P. R. China\\
\small{\tt Email:qiyan@cumt.edu.cn; xajin@xmu.edu.cn}
}

\begin{abstract}
Recently, Gross, Mansour and Tucker introduced the partial-dual genus polynomial of a ribbon graph as a generating function that enumerates the partial duals of the ribbon graph by genus. It is analogous to the extensively-studied polynomial in topological graph theory that enumerates by genus all embeddings of a given graph. To investigate the partial-dual genus polynomial one only needs to focus on bouquets, i.e. ribbon graphs with only one vertex. In this paper, we shall further show that the partial-dual genus polynomial of a bouquet essentially depends on the signed intersection graph of the bouquet rather than on the bouquet itself. That is to say the bouquets with the same signed intersection graph will have the same partial-dual genus polynomial. We then prove that the partial-dual genus polynomial of a bouquet contains non-zero constant term if and only if its signed intersection graph is positive and bipartite. Finally we consider a conjecture posed by Gross, Mansour and Tucker. that there is no orientable ribbon graph whose partial-dual genus polynomial has only one non-constant term, we give a characterization of non-empty bouquets whose partial-dual genus polynomials have only one term by consider non-orientable case and orientable case separately.
\end{abstract}

\begin{keyword}
Ribbon graph, partial-dual genus polynomial, bouquet, signed intersection graph, bipartite, complete.
\vskip0.2cm

\MSC 05C10\sep 05C30\sep 05C31\sep  57M15
\end{keyword}

\maketitle

\section{Introduction}
\noindent

The concept of partial duality was introduced in \cite{CG} by  Chmutov, and it, together with other partial twualities, has received ever-increasing attention, and their applications span topological graph theory,
knot theory, matroids/delta matroids, and physics. We assume that the readers are familiar with the basic knowledge of topological graph theory, see for example \cite{GT, Mohar} and in particular the ribbon graphs and partial duals, see for example \cite{BR2, CG,DJM, EM, MJ}. Let $G$ be a ribbon graph and $A\subseteq E(G)$. We denote by $G^{A}$ the partial dual of $G$ with respect to $A$.

Similar to the extensively-studied polynomial in topological graph theory that enumerates by genus all embeddings of a given graph. In \cite{GMT}, Gross, Mansour and Tucker introduced the partial-dual orientable genus polynomials for orientable ribbon graphs and the partial-dual Euler genus polynomials for arbitrary ribbon graphs.

\begin{definition}\label{def-1}\cite{GMT}
The \emph{partial-dual Euler genus polynomial} of any ribbon graph $G$ is the generating function
$$^{\partial}\varepsilon_{G}(z)=\sum_{A\subseteq E(G)}z^{\varepsilon(G^{A})}$$
that enumerates all partial duals of $G$ by Euler genus.
The \emph{partial-dual orientable genus polynomial} of an orientable ribbon graph $G$ is the generating function
$$^{\partial}\Gamma_{G}(z)=\sum_{A\subseteq E(G)}z^{\gamma(G^{A})}$$
that enumerates all partial duals of $G$ by orientable genus.
\end{definition}

Clearly, if $G$ is an orientable ribbon graph, then $^{\partial}\Gamma_{G}(z)={^{\partial}}\varepsilon_{G}(z^{\frac{1}{2}})$. Either $^{\partial}\varepsilon_{G}(z)$ or $^{\partial}\Gamma_{G}(z)$ may be referred to as a
partial-dual genus polynomial. A \emph{bouquet} is a ribbon graph with only one vertex. It is known that the partial-dual genus polynomial of any connected ribbon graph is equal to that of a bouquet, which is one of its partial duals. Hence it is natural to focus on bouquets.

In \cite{QYJ}, we introduced the notion of signed interlace sequences of bouquets and proved that two bouquets with the same signed interlace sequence have the same partial-dual Euler genus polynomial if the number of edges of the bouquets is less than 4 and two orientable bouquets with the same signed interlace sequence have the same partial-dual orientable genus polynomial if the number of edges of the bouquets is less than 5. As we observed in \cite{QYJ}, there are bouquets with the same signed interlace sequence but different partial-dual genus polynomials. The first purpose of this paper is to strengthen the notion of signed interlace sequences such that it can determine the partial-dual genus polynomial completely.

Intersection graphs (also called circle graphs) appear and are very useful in both graph theory and combinatorial knot theory \cite{GR}. For example, a characterization of those graphs that can be realized as intersection graphs is given by an elegant theorem of Bouchet \cite{Bo}. Signed interlace sequences of bouquets are exactly degree sequences of their signed intersection graphs. Based on a theorem of Chmutov and Lando \cite{CSL}, we shall prove that any two bouquets with the same signed intersection graph will have the same partial-dual genus polynomial.

Then we focus on signed intersection graphs, the intersection polynomial is introduced, a recursion of this polynomial is given and is used to compute intersection polynomials of paths and stars. We also prove that the intersection polynomial contains non-zero constant term if and only if the signed intersection graph is positive and bipartite.

In \cite{GMT}, Gross, Mansour and Tucker characterized connected ribbon graphs with constant polynomials, i.e., one of its partial dual is a tree. They also found examples of non-orientable ribbon graphs whose polynomials have only one (non-constant) term. The second purpose of this paper is to characterize bouquets whose partial-dual genus polynomials have only one term.

A bouquet is \emph{prime} if its intersection graph is connected. Then we will show that the partial-dual Euler genus polynomial of a prime non-orientable bouquet has only one non-constant term if and only if its intersection graph is trivial. For orientable ribbon graphs, they posed the following conjecture.
\begin{conjecture}\cite{GMT}\label{con-01}
There is no orientable ribbon graph having a non-constant partial-dual genus polynomial with only one non-zero coefficient.
\end{conjecture}

The conjecture is not true. In \cite{QYJ} we found an infinite family of counterexamples (see Proposition \ref{le-06}), whose intersection graphs are non-trivial complete graph of odd order. In this paper, we shall prove that Conjecture \ref{con-01} is actually true for all prime orientable bouquets except the family of counterexamples.

This paper is organized as follows. In Section 2, we consider the cyclic interlace sequences and an example is given to show that it can not determine the partial-dual genus polynomial. In Section 3, we recall the notion of mutant chord diagrams and a theorem of Chmutov and Lando on mutant chord diagrams and intersection graphs. In Section 4, we prove that the signed intersection graph can determine the partial-dual genus polynomial. In Section 5, we introduce the intersection polynomial and discuss its basic properties. In Section 6, we characterize all the bouquets whose partial-dual genus polynomials have only one term. In the final section, we pose several problems for further study.

\section{Cyclic interlace sequences and signed intersection graphs}
\noindent

Let $e$ be an edge of a ribbon graph $G$. If the vertex-disks at the ends of $e$ are distinct, we say that $e$ is {\it proper}.
If $e$ is a loop at the vertex disk $v$ and $e\cup v$ is homeomorphic
to a M\"obius band, then we call $e$ a {\it twisted loop}. Otherwise it is said to be an {\it untwisted loop} .

A \emph{signed rotation} of a bouquet is a cyclic ordering of the half-edges at the vertex and if the edge is an untwisted loop, then we give the same sign $+$ to the corresponding two half-edges, and give the different signs (one $+$, the other $-$) otherwise. The sign $+$ is always omitted. See Figure \ref{f0001} for an example. Sometimes we will use the signed rotation to represent the bouquet itself.
\begin{figure}[!htbp]
\begin{center}
\includegraphics[width=10cm]{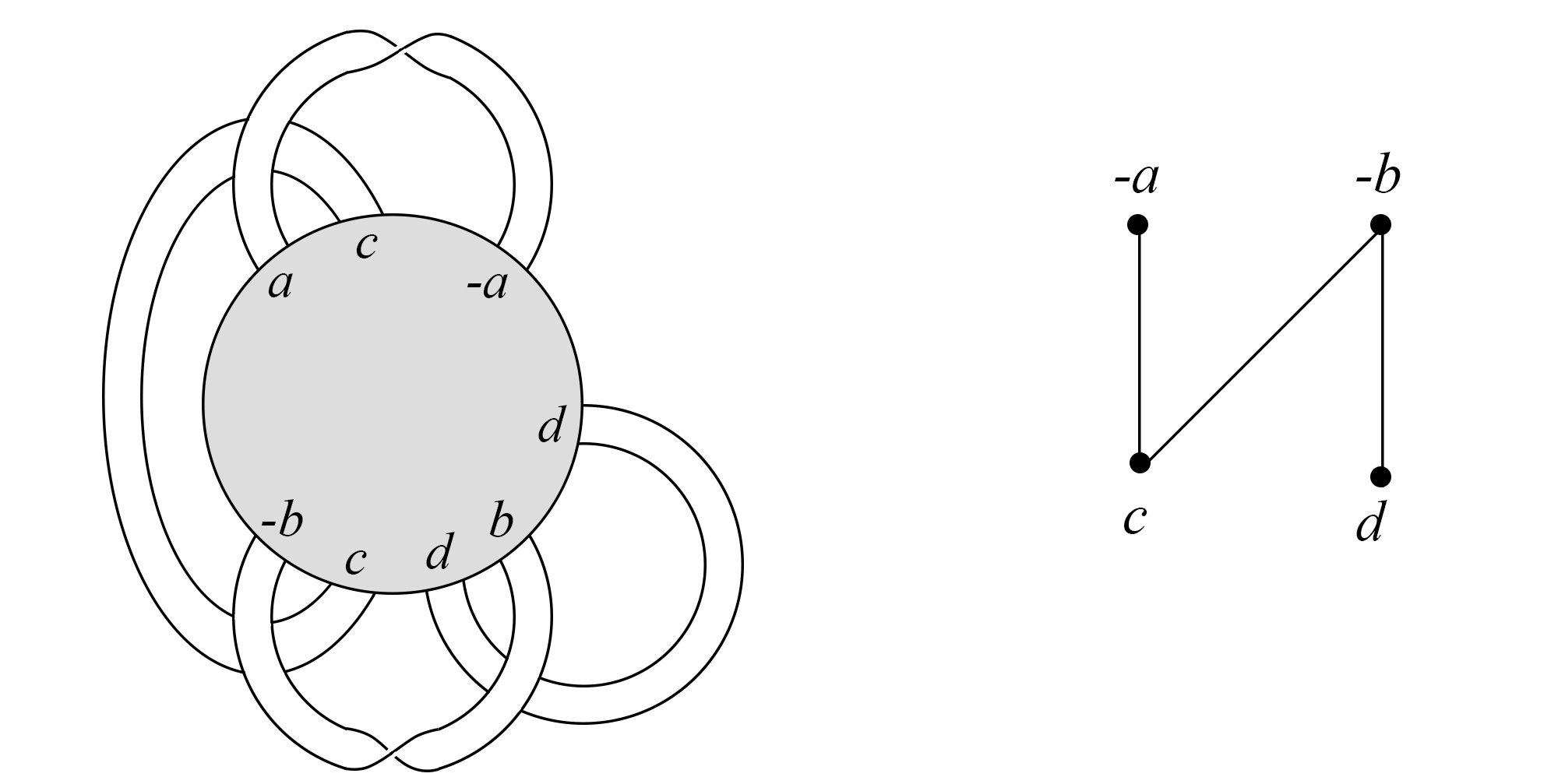}
\caption{ A bouquet with the signed rotation $(a, c, -a, d, b, d, c, -b)$ and its signed intersection graph.}
\label{f0001}
\end{center}
\end{figure}

Let $B$ be a bouquet and let $E(B)=\{e_{1}, \cdots, e_{n}\}$ and $e\in E(B)$.
The \emph{interlace number} of $e$, denoted by  $\alpha(e)$, is defined to be
the number of edges which are all interlaced with $e$.
We say that $\beta(e)$ is the \emph{signed interlace number} of $e$, where
\begin{eqnarray*}
\beta(e)=\left\{\begin{array}{ll}
                      \alpha(e), & \mbox{if}~e~\mbox{is an untwisted loop,}\\
                    -\alpha(e), & \mbox{if}~e~\mbox{is a twisted loop.}
                   \end{array}\right.
\end{eqnarray*}

The \emph{signed interlace sequence} \cite{QYJ} of the bouquet $B$, denoted by $\mathcal{S}(B)=(\beta(e_{1}), \cdots, \beta(e_{n}))$, is obtained by sorting the signed interlace number from small to large, where $\beta(e_1)\leq \beta(e_2)\leq\cdots\leq \beta(e_n)$. We first strengthen a signed interlace sequence by considering its cyclic rotation of a bouquet.

\begin{definition}
The \emph{cyclic interlace sequence} of a bouquet can be obtained from its cyclic rotation by replacing each entry with the signed interlace number of the corresponding edge.
\end{definition}

\begin{example}\label{ex-01}
Let $B_{1}$ be a bouquet with signed rotation $$(a, b, a, c, b, d, e, f, d, e, c, f)$$ and let $B_{2}$ be a bouquet with signed rotation $$(a, b, a, c, d, e, c, f, e, d, b, f).$$ It is easy to check that both bouquets $B_{1}$ and $B_{2}$
have the same cyclic interlace sequence $(1, 2, 1, 2, 2, 2, 2, 3, 2, 2, 2, 3)$. But $$^{\partial}\Gamma_{B_{1}}(z)=12z+44z^{2}+8z^{3}$$  and
$$^{\partial}\Gamma_{B_{2}}(z)=2+18z+36z^{2}+8z^{3}.$$
\end{example}

The {\it intersection graph} $I(B)$ of  a bouquet $B$ is the graph with vertex set $E(B)$ and in which two vertices $e$ and $f$ of $I(B)$ are adjacent if and only if their ends are met in the cyclic order $e\cdots f\cdots e\cdots f\cdots$ when traveling around the boundary of the unique vertex of $B$. Clearly, the degree sequence of the intersection graph is exactly the interlace sequence.
The {\it signed intersection graph} $SI(B)$ of a bouquet $B$ consists of $I(B)$ and a $+$ or $-$ sign at each vertex of $I(B)$ where the vertex corresponding to the untwisted loop of $B$ is signed $+$ and the vertex corresponding to the twisted loop of $B$ is signed $-$. See Figure \ref{f0001} for an example. A signed intersection graph is said to be \emph{positive} if each of its vertices is signed $+$. The following lemma is obvious.

\begin{lemma}\label{j1}
A bouquet $B$ is orientable if and only if its signed intersection graph $SI(B)$ is positive.
\end{lemma}

The signed intersection graphs of the bouquets $B_{1}$ and $B_{2}$ in Example \ref{ex-01} are shown in Figure \ref{f0002}. It shows that two bouquets with the same cyclic interlace sequence may have different signed intersection graphs.

\begin{figure}[!htbp]
\begin{center}
\includegraphics[width=9cm]{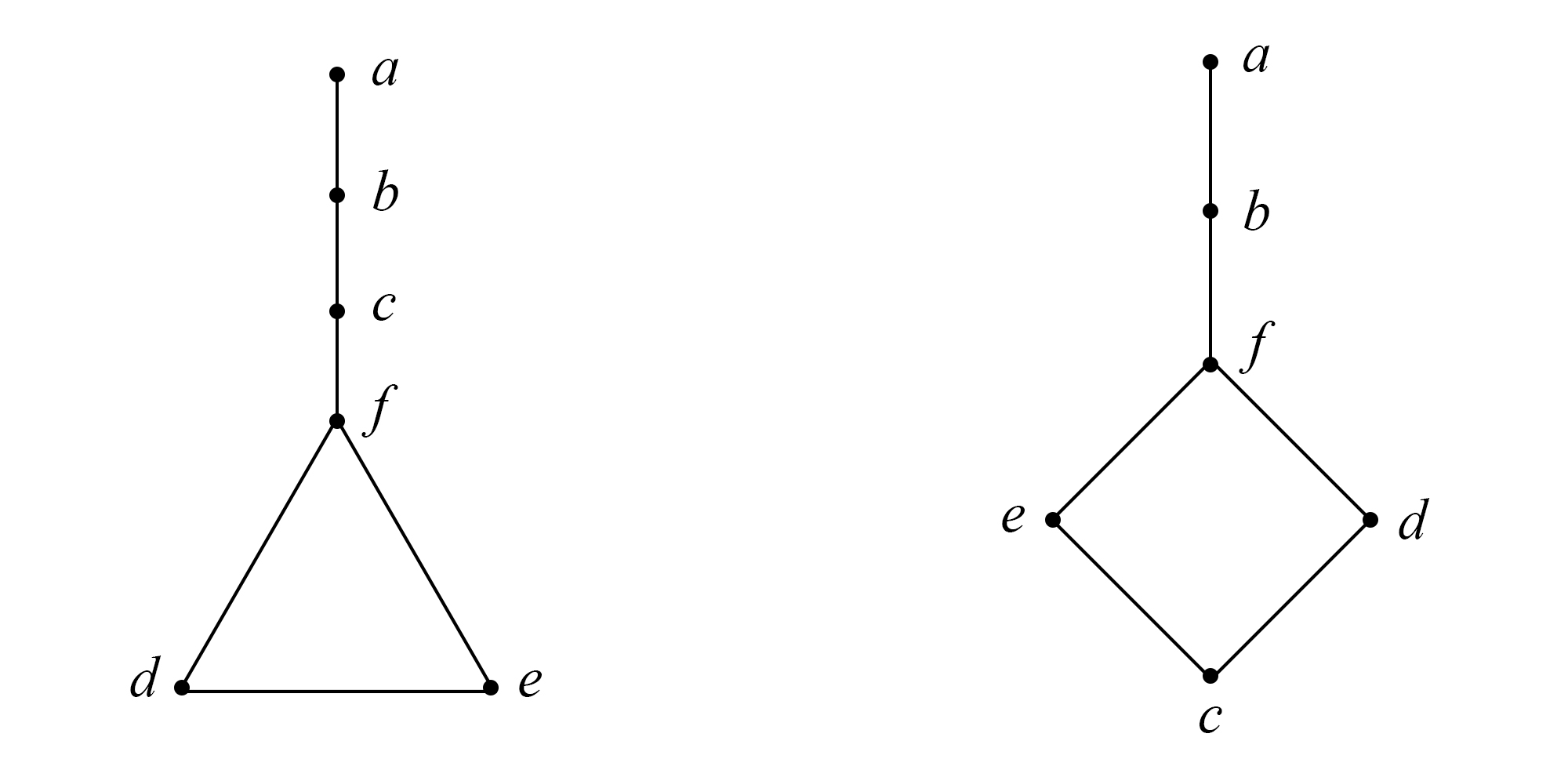}
\caption{$SI(a, b, a, c, b, d, e, f, d, e, c, f)$ and $SI(a, b, a, c, d, e, c, f, e, d, b, f)$.}
\label{f0002}
\end{center}
\end{figure}

In the following, we shall prove that signed intersection graphs can determine the partial-dual genus polynomial completely. In next section we will first recall mutants. 

\section{Mutants}
\noindent

In knot theory mutants are a pair of knots obtained from one to the other by rotating a tangle. Mutants are usually very difficult to distinguish by knot polynomials. 

A \emph{chord diagram} refers to a set of chords with distinct endpoints on a circle. A combinatorial analog of the tangle in mutant knots is a \emph{share}. A {\it share}  \cite{CSL} in a chord diagram is a union of two arcs of the outer circle and chords ending on them possessing the following property: each chord one of whose ends belongs to these arcs has both ends on these arcs. A {\it mutation} \cite{CSL} of a chord diagram is another chord diagram obtained by a rotation of a share about one of the three axes. Note that the composition of rotations about two of the three axes will be exactly the rotation about the third axis. Two chord diagrams are said to be {\it mutant} \cite{CSL} if they can be transformed into one another by a sequence of mutations.

\begin{theorem}\cite{CSL}\label{le-02}
Two chord diagrams have the same intersection graph if and only if they are mutant.
\end{theorem}

For the details we refer the reader to \cite{CSL}.  Mutant can be defined for bouquets similarly. Suppose $P=p_{1}p_{2}\cdots p_{k}$ is a \emph{string} and $P^{-1}=p_{k}p_{k-1}\cdots p_{1}$ is called the \emph{inverse} of $P$.

\begin{definition}\label{re}
Let $B$ be a bouquet with signed rotation $(MPNQ)$ where both labels of each edge must belong  to $MN$ or both not.
A {\it mutation} of $B$ is another bouquet with signed rotation $(M^{-1}PN^{-1}Q)$ or $(NPMQ)$. Two bouquets are said to be {\it mutant} if they can be transformed into one another by a sequence of mutations.
\end{definition}

In Definition \ref{re}, either $M, N, P$ or $Q$ can be empty.

\begin{corollary}\label{le-03}
Two bouquets have the same signed intersection graph if and only if they are mutant.
\end{corollary}

\begin{proof}
Obviously, mutations preserve the signed intersection graphs of bouquets, hence if two bouquets are mutant, they have the same signed intersection graph. Conversely, if two bouquets have the same signed intersection graph, by Theorem \ref{le-02}, they are related by a sequence of mutations.
\end{proof}

In the next section, we will show that the signed intersection graphs can determine the partial-dual genus polynomial completely.

\section{First Main Theorem}
\noindent

Now we state our first main theorem as follows.

\begin{theorem}\label{main-1}
If two bouquets $B_{1}$ and $B_{2}$ have the same signed intersection graph, then $^{\partial}\varepsilon_{B_{1}}(z)={^{\partial}}\varepsilon_{B_{2}}(z)$.
\end{theorem}

Recall that the contraction $G/e$ of the edge $e$ in the ribbon graph $G$ is defined by the equation $G/e:=G^{e}-e$. We denote by $G/A$ the ribbon graph obtained from $G$ by contracting each edge of $A\subseteq E(G)$ and then $G/A=G^{A}-A.$  It is an important observation \cite{EM,GJY} that the operation of the contraction does not change the number of boundary components. Let $v(G), e(G)$ and $f(G)$ denote the number of vertices, edges and boundary components of a ribbon graph $G$, respectively. To prove Theorem \ref{main-1}, we need three lemmas.

\begin{lemma}\label{J2}
Let $B$ be a bouquet. Then the Euler genus $\varepsilon(B)$ is given by the equation:
\begin{eqnarray}
\varepsilon(B)=1+e(B)-f(B).
\end{eqnarray}
\end{lemma}

\begin{proof}
Recall that if $G$ is a connected ribbon graph then $2-\varepsilon(G)=v(G)-e(G)+f(G)$. The Lemma then follows from $v(B)=1$.
\end{proof}

\begin{lemma}\label{le-04}
If two bouquets $B_{1}$ and $B_{2}$ have the same signed intersection graph, then $\varepsilon(B_{1})=\varepsilon(B_{2})$.
\end{lemma}

\begin{proof}

By Corollary \ref{le-03}, we can assume that $B_{1}$ can be transformed into $B_{2}$ by a mutation. Let $B_{1}=(MPNQ)$. Then $B_{2}=(M^{-1}PN^{-1}Q)$ or $B_{2}=(NPMQ)$ as in Figure \ref{f0003}.
Assume that $B_{3}=(M^{-1}PN^{-1}Q)$ and $B_{4}=(NPMQ)$.
\begin{figure}[!htbp]
\begin{center}
\includegraphics[width=12cm]{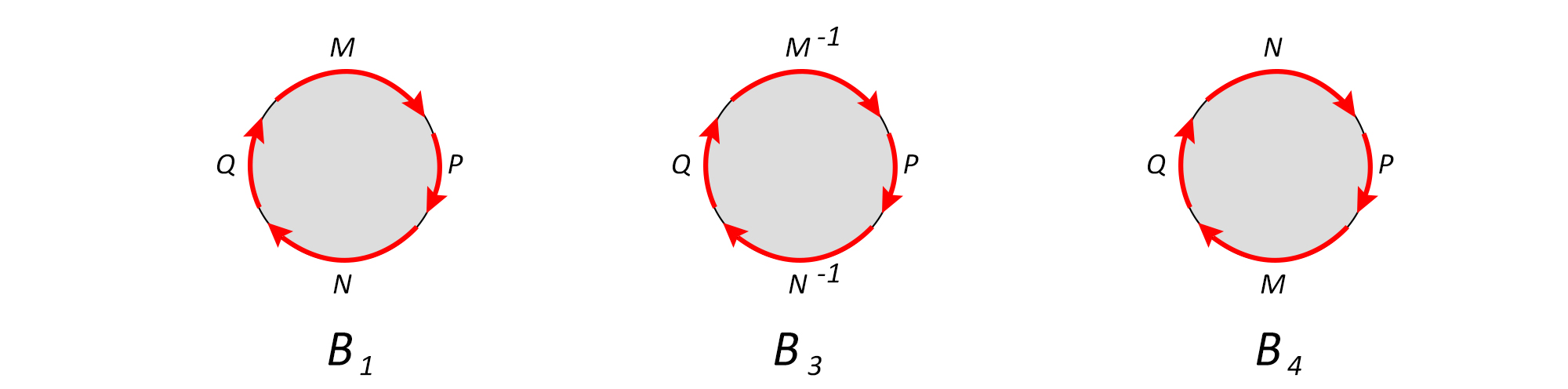}
\caption{The bouquets $B_{1}, B_{3}$ and $B_{4}$.}
\label{f0003}
\end{center}
\end{figure}
By Lemma \ref{J2}, it suffices to prove that $f(B_{1})=f(B_{3})=f(B_{4})$.

Suppose that $G_{1}=(MfPfNeQe), G_{3}=(M^{-1}fPfN^{-1}eQe)$ and $G_{4}=(NfPfMeQe)$ as in Figure \ref{f0004}.
\begin{figure}[!htbp]
\begin{center}
\includegraphics[width=13cm]{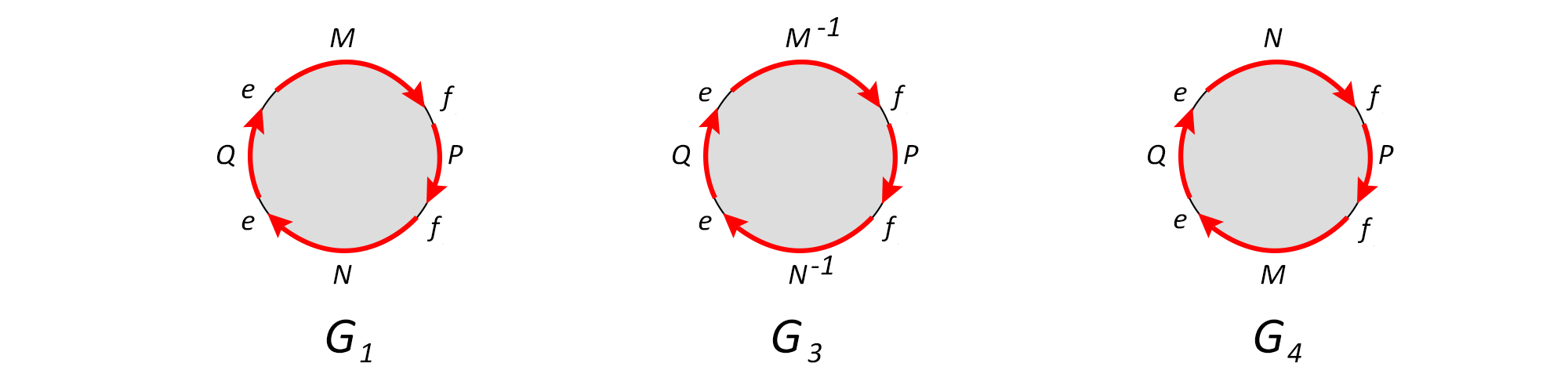}
\caption{The bouquets $G_{1}, G_{3}$ and $G_{4}$.}
\label{f0004}
\end{center}
\end{figure}

Since  $B_{i}=G_{i}-\{e, f\}=({G_{i}}^{\{e,f\}})^{\{e,f\}}-\{e, f\}={G_{i}}^{\{e,f\}}/{\{e,f\}}$ for $i\in \{1, 3, 4\}$ and contraction does not change the number of boundary components, it follows that $f(B_{i})=f({G_{i}}^{\{e,f\}})$.
For the ribbon graph ${G_{i}}^{\{e,f\}}$, arbitrarily orient the boundary of $e$ and place an arrow on each of the
two arcs where $e$ meets vertices of ${G_{i}}^{\{e,f\}}$ such that the directions of these arrows follow the orientation of the boundary of $e$ and label the two arrows with $e'$ and $e''$. The same operating can be drawn for $f$ and label the two arrows with $f'$ and $f''$.
Let ${B_{i}}'$ denote the ribbon graph obtained from ${G_{i}}^{\{e,f\}}$ by deleting the vertices $v_{P}, v_{Q}$ together with all the edges incident with $v_{P}, v_{Q}$, but reserving the marking arrows  $e''$ and $f''$ as in Figure \ref{f0005}.
\begin{figure}[!htbp]
\begin{center}
\includegraphics[width=13cm]{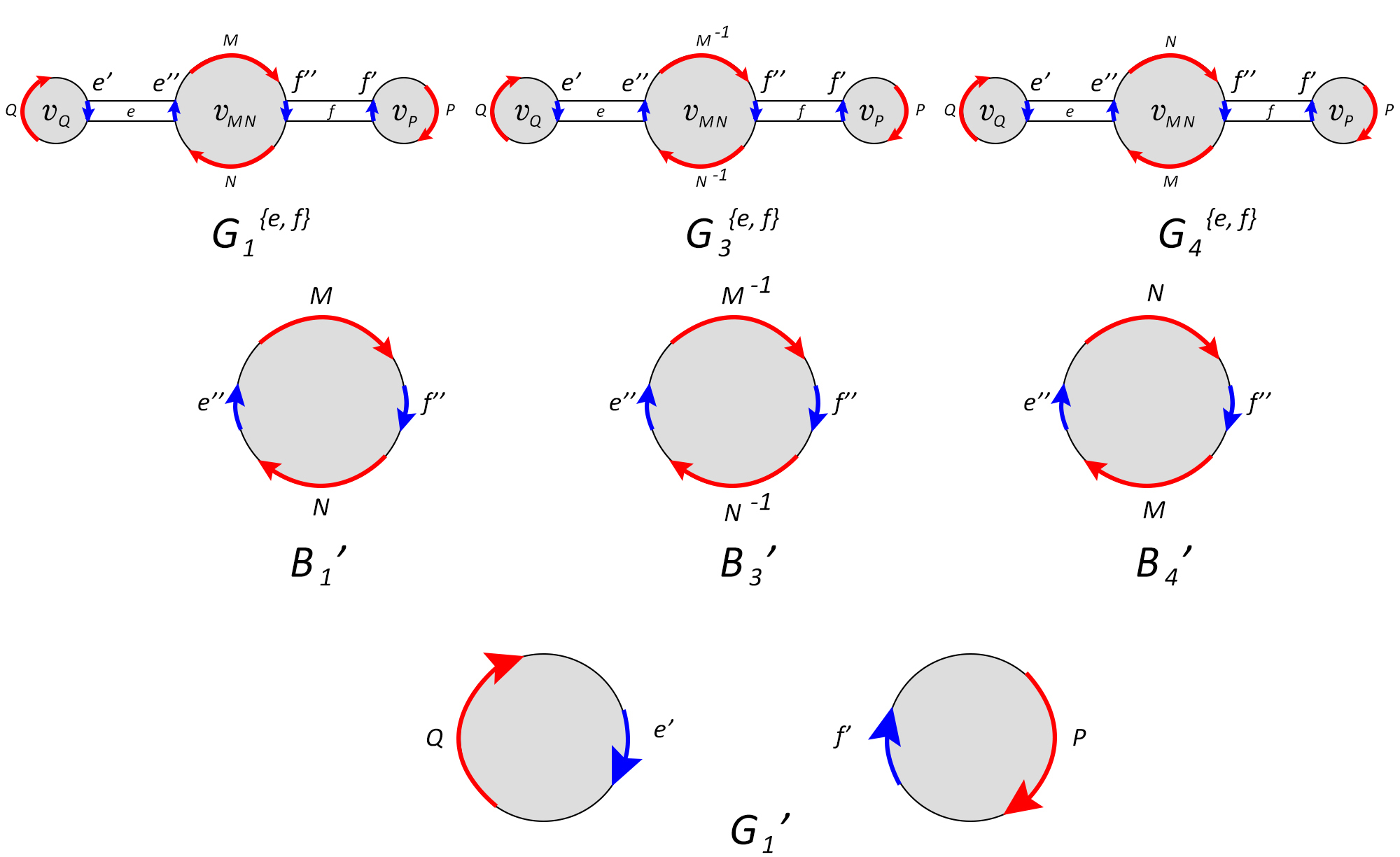}
\caption{The ribbon graphs ${G_{i}}^{\{e,f\}}$ and ${B_{i}}'$ for $i\in \{1, 3, 4\}$ and ${G_{1}}'$ .}
\label{f0005}
\end{center}
\end{figure}
Since both labels of each edge must belong to $MN$ or both not, this results in a bouquet, with exactly two labelled arrows $e''$ and $f''$ on its boundary of the vertex and these marking arrows only indicate the positions and no other significance.
Note that if we ignore the two labelled arrows $e''$ and $f''$, the bouquets  ${B_{1}}'$, ${B_{3}}'$ and ${B_{4}}'$ are equivalent. Hence $f({B_{1}}')=f({B_{3}}')=f({B_{4}}').$
Similarly, let ${G_{i}}'$ denote the ribbon graph obtained from ${G_{i}}^{\{e,f\}}$ by deleting the vertex $v_{MN}$ together with all the edges incident with $v_{MN}$,  but reserving the marking arrows  $e'$ and $f'$. This results in a ribbon graph, with exactly two labelled arrows $e'$ and $f'$ on the boundaries of  $v_{P}$ and $v_{Q}$ as in Figure \ref{f0005}. Note that ${G_{1}}'={G_{3}}'={G_{4}}'$.
Obviously, we can recover the boundaries of ${G_{i}}^{\{e,f\}}$ from ${G_{i}}'$ and ${B_{i}}'$ as follows: draw a line segment from the head of $e'$  to the tail of $e''$, and a line segment from the head of $e''$ to the tail of $e'$. The same operating is applied to $f'$ and $f''$. We observe that

({\bf a}) If $e''$ and $f''$ are contained in different boundary components of ${B_{1}}'$, then $e''$ and $f''$ are also contained in different boundary components of ${B_{3}}'$ and ${B_{4}}'$.

({\bf b}) If $e''$ and $f''$ are contained in the same boundary component of ${B_{1}}'$, then $e''$ and $f''$ are also contained in the same boundary component of ${B_{3}}'$ and ${B_{4}}'$.
The arrows $e''$ and $f''$ are called consistent (inconsistent) in ${B_{1}}'$  if these two arrows are consistent (inconsistent) on the boundary component.
We can also observe that if $e''$ and $f''$ are consistent (inconsistent) in ${B_{1}}'$, then $e''$ and $f''$ are also consistent (inconsistent) in ${B_{3}}'$ and ${B_{4}}'$.

If $e'$ and $f'$ are contained in the same boundary component of ${G_{1}}'$ and $e'$,  $f'$ are consistent in ${G_{1}}'$, then there are three cases as following.
\begin{enumerate}
  \item If $e''$ and $f''$ are contained in different boundary components of ${B_{1}}'$, then by (a)
  $$f({G_{1}}^{\{e,f\}})=f({G_{3}}^{\{e,f\}})=f({G_{4}}^{\{e,f\}})=f({G_{1}}')+f({B_{1}}')-2$$
as in Figure \ref{f0006}.
\begin{figure}[!htbp]
\begin{center}
\includegraphics[width=12cm]{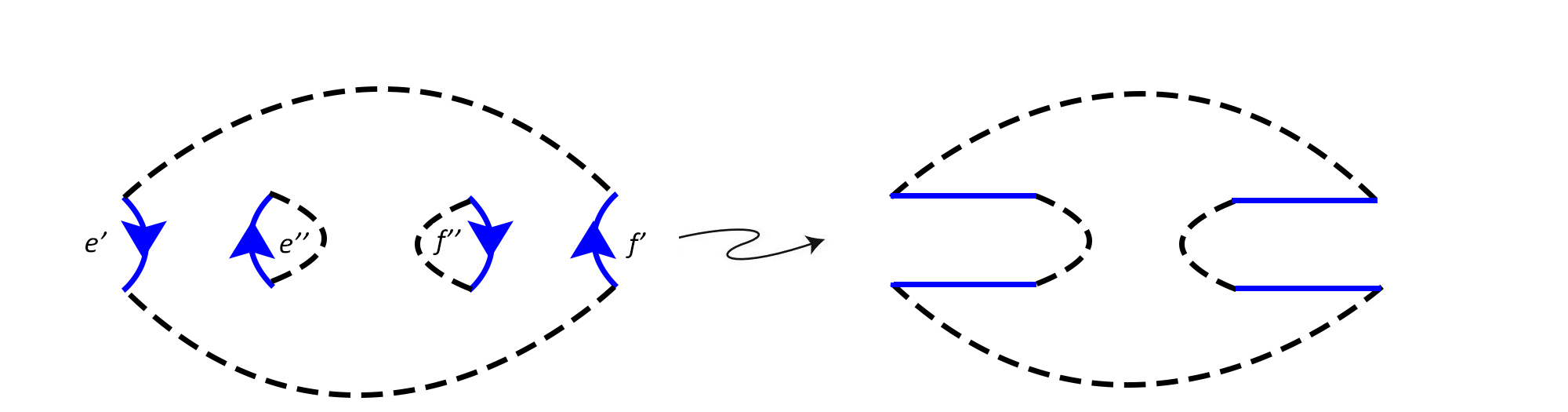}
\caption{Case 1.}
\label{f0006}
\end{center}
\end{figure}
  \item If $e''$ and $f''$ are contained in the same boundary component of ${B_{1}}'$ and $e''$, $f''$ are consistent in ${B_{1}}'$, then by (b)
  $$f({G_{1}}^{\{e,f\}})=f({G_{3}}^{\{e,f\}})=f({G_{4}}^{\{e,f\}})=f({G_{1}}')+f({B_{1}}')$$
 \begin{figure}[!htbp]
\begin{center}
\includegraphics[width=12cm]{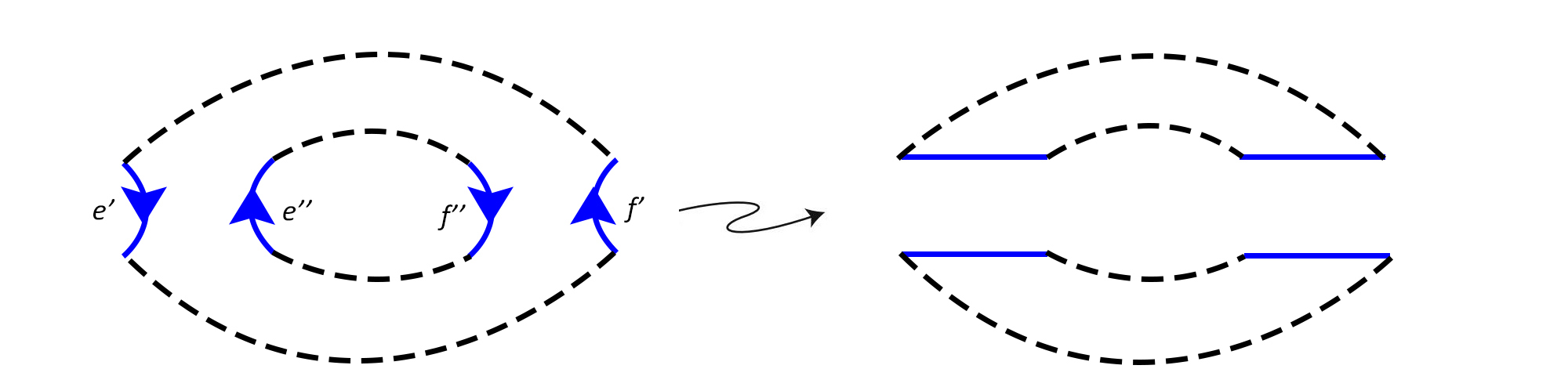}
\caption{Case 2.}
\label{f0007}
\end{center}
\end{figure}
as in Figure \ref{f0007}.
   \item If $e''$ and $f''$ are contained in the same boundary component of ${B_{1}}'$ and $e''$, $f''$ are inconsistent in ${B_{1}}'$, then by (b)
  $$f({G_{1}}^{\{e,f\}})=f({G_{3}}^{\{e,f\}})=f({G_{4}}^{\{e,f\}})=f({G_{1}}')+f({B_{1}}')-1$$
   \begin{figure}[!htbp]
\begin{center}
\includegraphics[width=12cm]{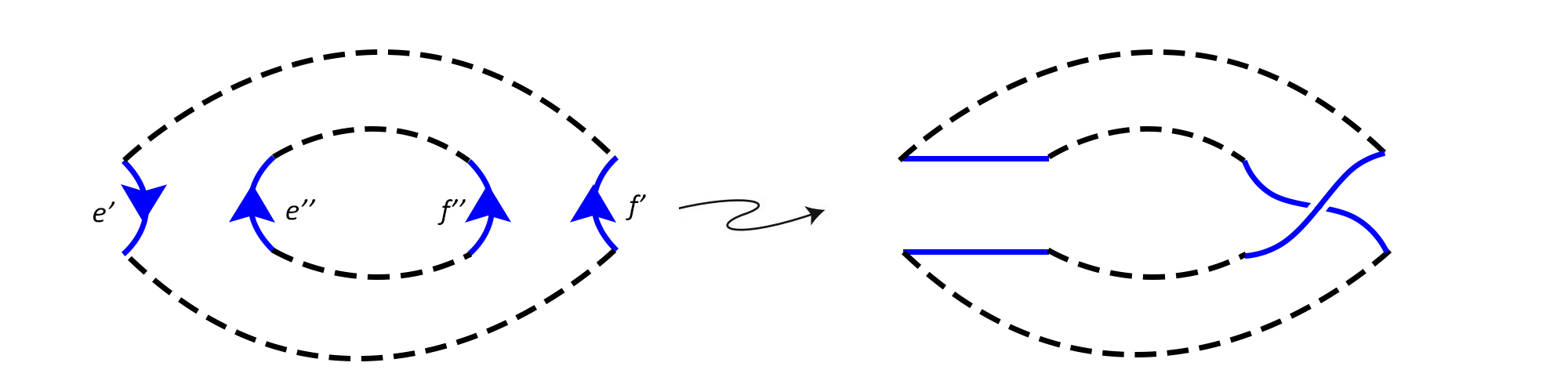}
\caption{Case 3.}
\label{f0008}
\end{center}
\end{figure}
as in Figure \ref{f0008}.
\end{enumerate}
Similar arguments apply to the case $e'$ and $f'$ are contained in different boundary components of ${G_{1}}'$ or $e'$ and $f'$ are contained in the same boundary component of ${G_{1}}'$ and $e'$, $f'$ are inconsistent in ${G_{1}}'$.
\end{proof}

\begin{lemma}\label{le-01}\cite{GMT}
Let $B$ be a bouquet, and let $A\subseteq E(B)$. Then 
\begin{eqnarray}
\varepsilon(B^{A})=\varepsilon(A)+\varepsilon(A^{c}),
\end{eqnarray}
where $A^{c}=E(B)-A$.
\end{lemma}

\vskip0.3cm

\noindent{\bf Proof of Theorem \ref{main-1}}

For any subset $A_1$ of edges of $B_{1}$, we denoted its corresponding vertex subset of $SI(B_{1})$ also by $A_1$.
Let $SI(B_{1})[A_1]$ denote the induced subgraph of $SI(B_{1})$ by the vertex subset $A_1$. Since $SI(B_{1})=SI(B_{2})$, there is a corresponding subset $A_2$ of vertices of $SI(B_{2})$ such that $SI(B_{1})[A_1]=SI(B_{2})[A_2]$ and $SI(B_{1})[A_1^{c}]=SI(B_{2})[A_2^{c}]$. It follows that $\varepsilon(A_1)=\varepsilon(A_2)$ and
$\varepsilon(A_1^{c})=\varepsilon(A_2^{c})$ by Lemma \ref{le-04}. Hence, $\varepsilon({B_{1}}^{A_1})=\varepsilon({B_{2}}^{A_2})$ by Lemma \ref{le-01}. Thus $^{\partial}\varepsilon_{B_{1}}(z)={^{\partial}}\varepsilon_{B_{2}}(z)$.

\begin{remark}\label{rem1}
Two bouquets with different signed intersection graphs may have the same partial-dual genus polynomial. For example, let
$B_{1}=(1, 2, -1, 2)$ and $B_{2}=(1, 2, -1, -2)$. Obviously, ${^{\partial}}\varepsilon_{B_{1}}(z)={^{\partial}}\varepsilon_{B_{2}}(z)=2z+2z^{2}$, see also \cite{QYJ},  but the signed intersection graphs of $B_{1}$ and $B_{2}$ are different. In fact, $B_2=B_1^{e_1}$.
\end{remark}

\section{Intersection polynomials}
\noindent

A graph $SG$ with a $+$ or $-$ sign at each vertex is said to be \emph{signed intersection graph} if there exists a bouquet $B$ such that  $SG=SI(B)$. 
The \emph{intersection polynomial}, $IP_{SG}(z)$, of a signed intersection graph $SG$ is defined by $IP_{SG}(z):={^{\partial}}\varepsilon_{B}(z)$, where $B$ is a bouquet such that $SG=SI(B)$. By Theorem \ref{main-1}, it is well-defined.

\begin{theorem}\label{le-10}
Let $SG$ be a signed intersection graph and $v_{1}, v_{2}\in V(SG)$.
If $v_{1}, v_{2}$ are adjacent and the degree of $v_{1}$ is 1 and the sign of $v_{1}$ is positive, then
$$IP_{SG}(z)=IP_{SG-v_{1}}(z)+(2z^{2})IP_{SG-v_{1}-v_{2}}(z).$$
\end{theorem}

\begin{proof}
Let $B$ be a bouquet satisfying $SG=SI(B)$. We have $IP_{SG}(z)={^{\partial}}\varepsilon_{B}(z)$.
Note that $v_{1}, v_{2}$ correspond to two edges of $B$, we denote them by $e_{1}$ and $e_{2}$, respectively.
Since the degree of $v_{1}$ is 1 and the sign of $v_{1}$ is positive, it follows that $e_{1}$ is an untwisted loop and  for any $e\in E(B)-e_{1}-e_{2}$,
the ends of $e$ are therefore on $\alpha$ and $\beta$, or $\gamma$ and $\theta$ (otherwise it interlaces $e_{1}$) as shown in Figure \ref{f0009}.
\begin{figure}[!htbp]
\begin{center}
\includegraphics[width=12cm]{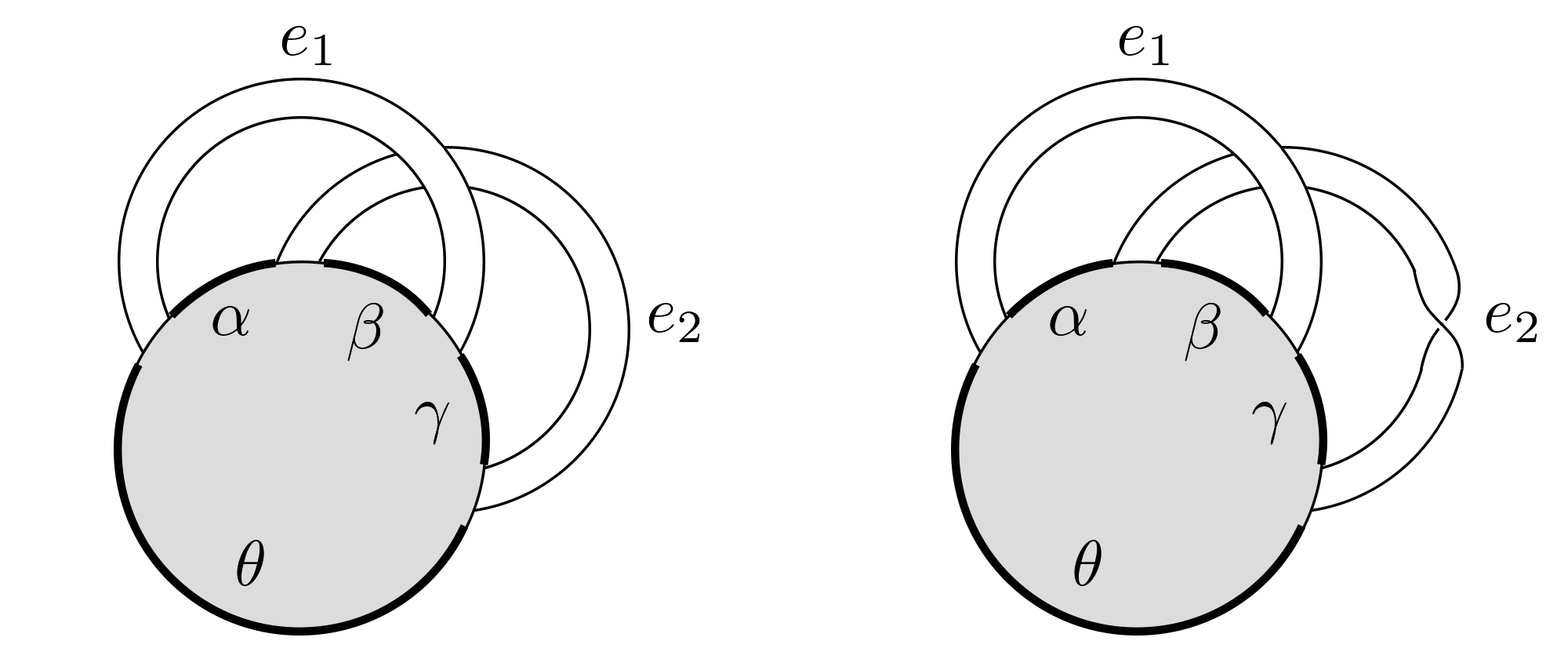}
\caption{Two cases for the bouquet $B$ in the proof of Theorem \ref{le-10}.}
\label{f0009}
\end{center}
\end{figure}
Let $A$ be any subset of $E(B)$. We partition the possible subsets $A$ into two types:
\begin{enumerate}
  \item $\tau_{1}$: those for which one of $e_{1}, e_{2}$ is in $A$ and the other is in $A^{c}$;
  \item $\tau_{2}$: those for which $e_{1}, e_{2}$ are both in $A$ or both in $A^{c}$.
\end{enumerate}
Then
$${^{\partial}}\varepsilon_{B}(z)=\Sigma_{A\in \tau_{1}}z^{\varepsilon(B^{A})}+\Sigma_{A\in \tau_{2}}z^{\varepsilon(B^{A})}.$$

Let $D\subseteq E(B-e_1)$. Then $D^c=B-e_1-D$. If $e_2\in D$, take $A=D$ and $A^c=D^c\cup e_1$; if $e_2\notin D$, take  $A=D\cup e_1$ and $A^c=D^c$. A 1-1 correspondence between the set of subsets of $E(B-e_1)$ and $\tau_{1}$ is established. Furthermore, it is not difficult to see that $\varepsilon(D)=\varepsilon(A)$ and $\varepsilon(A^{c})=\varepsilon(D^{c})$ for each case. By Lemma \ref{le-01} we have 
\begin{eqnarray*}
\Sigma_{A\in \tau_{1}}z^{\varepsilon(B^{A})}={^{\partial}}\varepsilon_{B-e_{1}}(z).
\end{eqnarray*}

Let $D\subseteq E(B-e_{1}-e_{2})$. Then $D^c=B-e_1-e_2-D$. Take $A=D\cup \{e_1,e_2\}$ and $A^c=D^c$. Clearly $\varepsilon(A^{c})=\varepsilon(D^{c})$ and it is not difficult to see that $f(A)=f(D)$, hence  $\varepsilon(A)=\varepsilon(D)+2$. Thus, we have

\begin{eqnarray*}
\Sigma_{A\in \tau_{2}}z^{\varepsilon(B^{A})}&=&2\Sigma_{\{e_1,e_2\}\subseteq A\in \tau_{2}}z^{\varepsilon(B^{A})}\\
&=&(2z^{2})~{^{\partial}}\varepsilon_{B-e_{1}-e_{2}}(z).
\end{eqnarray*}
 
Hence, $
{^{\partial}}\varepsilon_{B}(z)={^{\partial}}\varepsilon_{B-e_{1}}(z)+(2z^{2})~{^{\partial}}\varepsilon_{B-e_{1}-e_{2}}(z), 
$
i.e. $IP_{SG}(z)=IP_{SG-v_{1}}(z)+(2z^{2})IP_{SG-v_{1}-v_{2}}(z)$.

\end{proof}

\begin{example}
Let $S_{n}$ be a positive star which is a complete bipartite graph whose vertex set can be partitioned into two subsets $X$ and $Y$
so that every edge has one end in $X$ and the other end in $Y$ with $|X|=1$ and $|Y|=n$. Then 
we have initial condition $IP_{S_{1}}(z)=2+2z^{2}$, and by Theorem \ref{le-10}, the recursion
$$IP_{S_{n+1}}(z)=IP_{S_{n}}(z)+2^{n+1}z^{2}.$$
Then it is easy to obtain that 
\begin{eqnarray}
IP_{S_{n}}(z)=(2^{n+1}-2)z^{2}+2.
\end{eqnarray}
\end{example}

\begin{example}
Let $P_{n}$ be a positive path with $n$ vertices. Then 

$IP_{P_{1}}(z)=2;$

$IP_{P_{2}}(z)=2+2z^{2};$

$IP_{P_{n+2}}(z)=IP_{P_{n+1}}(z)+2z^{2}IP_{P_{n}}(z).$

\end{example}

\begin{theorem}
Let $SG$ be a signed intersection graph. Then $IP_{SG}(z)$ contains non-zero constant term if and only if $SG$ is positive and bipartite.
\end{theorem}

\begin{proof}
Let $B$ be a bouquet satisfying $SG=SI(B)$. We know $IP_{SG}(z)={^{\partial}}\varepsilon_{B}(z)$. Since $IP_{SG}(z)$ contains non-zero constant term,  it follows that $B$ is a partial dual of a plane ribbon graph. According to  the property that partial duality preserving orientability,  we have $B$ is orientable and hence $SG$ is positive. Suppose that $SG$ is not bipartite. Then $SG$ contains an odd cycle $C$. We denote by $D$ the edge subset of $B$ corresponding to vertices of $C$. It is obvious that deleting edges can not increase Euler genus. Then for any subset $A$ of $E(B)$, we have $\varepsilon(A\cap D)\leqslant \varepsilon(A)$, $\varepsilon(A^{c}\cap D)\leqslant \varepsilon(A^{c})$. Note that there are two loops $e, f\in A\cap D$ or  $e, f\in A^{c}\cap D$ such that
 their ends are met in the cyclic order $e\cdots f\cdots e\cdots f\cdots$ when traveling round the boundary of the unique vertex of $B$. Then $\varepsilon(A\cap D)+\varepsilon(A^{c}\cap D)>0$. Thus $\varepsilon(B^{A})=\varepsilon(A)+\varepsilon(A^{c})>0$, a contradiction.

Conversely, if $SG$ is bipartite and non-trivial, then its vertex set can be partitioned into two subsets $X$ and $Y$ so that
every edge of $SG$ has one end in $X$ and the other end in $Y$. For these two subsets $X$ and $Y$ of the vertex set of $SG$,
we denoted these two corresponding edge subsets of $B$ also by $X$ and $Y$. Obviously, $X\cup Y=E(B), X\cap Y=\emptyset$ and $\varepsilon(X)=\varepsilon(Y)=0$. Thus $\varepsilon(B^{X})=0$ by Lemma \ref{le-01}. Hence
 ${^{\partial}}\varepsilon_{B}(z)$ (hence, $IP_{SG}(z)$) contains non-zero constant term.
\end{proof}

\section{Second Main Theorem}
\noindent

Gross, Mansour and Tucker \cite{GMT} discussed the simplest partial-dual genus polynomial, i.e., a constant polynomial and found examples of non-orientable ribbon graphs whose polynomials have only one non-constant term.
They proved:
\begin{proposition} \cite{GMT}
Let $G$ be a connected ribbon graph. Then ${^{\partial}}\varepsilon_{G}(z)=2^{e(G)}$ if and only if there is a subset $A\subseteq E(G)$ such that $G^{A}$ is a tree.
\end{proposition}

and

\begin{proposition} \cite{GMT}
For any $n>0$ and any $m\geq n$, there is a non-orientable ribbon graph $G$ such that ${^{\partial}}\varepsilon_{G}(z)=2^{m}z^{n}$.
\end{proposition}

For orientable ribbon graphs, Gross, Mansour and Tucker posed Conjecture \ref{con-01} and we found an infinite family of counterexamples in \cite{QYJ}. Let $B_{t}$ be a bouquet with the signed rotation $(1, 2, 3, \cdots, t, 1, 2, 3, \cdots, t)$.

\begin{proposition}\label{le-06}\cite{QYJ}
Let $t$ be a positive integer. Then
\begin{eqnarray*}
^{\partial}\varepsilon_{B_{t}}(z)=\left\{\begin{array}{ll}
                    2^{t}z^{t-1}, & \mbox{if}~t~\mbox{is odd,}\\
                    2^{t-1}z^{t}+2^{t-1}z^{t-2}, & \mbox{if}~t~\mbox{is even.}
                   \end{array}\right.
\end{eqnarray*}
\end{proposition}

Note that $B_3,B_5,B_7,\cdots$ is an infinite family of counterexamples to Conjecture \ref{con-01}. The purpose of this section is to characterize bouquets whose partial-dual genus polynomial has only one non-constant term.

\subsection{Prime bouquets and our result}
\noindent

Let $P, Q$ be two ribbon graphs, we denote by $P\vee Q$ the \emph{ribbon-join} of $P$ and $Q$. Note that in general the ribbon-join is not unique.  A ribbon graph is called \emph{empty} if it has no edges.  We say that $G$ is \emph{prime}, if there don't exist non-empty ribbon subgraphs $G_{1}, \cdots,  G_{k}$ of $G$ such that $G=G_{1}\vee\cdots \vee G_{k}$ where $k\geq 2$. 
Clearly, we have
\begin{lemma}
A bouquet $B$ is prime if and only if its intersection graph $I(B)$ is connected.
\end{lemma}

Let $B_{\overline{1}}=(1, -1)$ be the non-orientable bouquet with only one edge. Let $\mathcal{B}=\{B_{\overline{1}}, B_1, B_3, B_5,\cdots\}$. Now we are in a position to state our second main theorem as follows.

\begin{theorem}\label{main2}
Let $B$ be a non-empty bouquet. Then 

$$
{^{\partial}}\varepsilon_{B}(z)=2^{e(B)}z^{b}\Longleftrightarrow B=B_{t_{1}}\vee\cdots \vee B_{t_{k}}, 
$$
where $k\geq 1$ and $B_{t_{i}}\in \mathcal{B}$ for $1\leqslant i \leqslant k$. Furthermore, if the number of the prime factors $B_{\overline{1}}$ in $B$ is $k_2$, then $b=e(B)-k+k_2$.
\end{theorem}

Note that the signed intersection graph of $B_{\overline{1}}$ is a negative isolated vertex and the signed intersection graph of $B_{2i+1}$ is a positive complete graph of order $2i+1$. In fact $B_{2i+1}$ is the only bouquet whose signed intersection graph is a positive complete graph of order $2i+1$. Restate Theorem \ref{main2} in the language of signed intersection graphs and intersection polynomial, we have

\begin{corollary}
Let $SG$ be a signed intersection graph. Then $IP_{SG}(z)=2^{v(SG)}z^{b}$ if and only if
each component of $SG$ is complete graph of odd order and each vertex, except isolated vertex, has positive sign.
\end{corollary}

It is easy to see that $^{\partial}\varepsilon_{B_1}(z)=2$ and $^{\partial}\varepsilon_{B_{\overline{1}}}(z)=2z$.  To prove Theorem \ref{main2}, we shall use the following lemma.

\begin{lemma}\label{le-08}\cite{GMT}
Let $G=G_{1}\vee G_{2}$. Then
\begin{eqnarray}
^{\partial}\varepsilon_{G}(z)={^{\partial}}\varepsilon_{G_1}(z)~^{\partial}\varepsilon_{G_2}(z).
\end{eqnarray}
\end{lemma}

It suffices to show that among all prime non-orientable bouquets there is only $B_{\overline{1}}$ whose partial-dual genus polynomial has one (non-constant) term and among all prime orientable bouquets there are only $B_{1}, B_3, B_5,\cdots$ whose partial-dual genus polynomials have one term.

Let $G^{*}$ denote the (full) dual  of a ribbon graph $G$. Corresponding to each edge $e$ of
$G$ there is an edge $e^{*}$ of $G^{*}$.
We view each ribbon as an oriented rectangle, then the opposing two sides lying on face-disks are called {\it ribbon-sides} \cite{GMT}. We need the following lemma.

\begin{lemma}\cite{GMT}\label{le-05}
Let $G$ be a ribbon graph and $e\in E(G)$. Then $\varepsilon(G)=\varepsilon(G^{e})$ if and only if
\begin{eqnarray*}
\left\{\begin{array}{ll}
                     e^{*}~\mbox{is proper in}~G^{*}, & \mbox{if}~e~\mbox{is an untwisted loop,}\\
                     e^{*}~\mbox{is an untwisted loop in}~G^{*}, & \mbox{if}~e~\mbox{is proper,}\\
                     e^{*}~\mbox{is a twisted loop in}~G^{*}, & \mbox{if}~e~\mbox{is a twisted loop.}
                   \end{array}\right.
\end{eqnarray*}
\end{lemma}

\subsection{Non-orientable case}
\noindent

\begin{theorem}\label{le-07}
Let $B$ be a prime non-orientable bouquet. Then ${^{\partial}}\varepsilon_{B}(z)=2^{e(B)}z^{b}$ if and only if $B=B_{\overline{1}}$.
\end{theorem}
\begin{proof}
The sufficiency is easily verified by calculation. For necessity,  since $B$ is non-orientable, we may assume that
$e(B)\geqslant 2$.

{\bf Claim 1. } $B$ does not contain a bouquet with signed rotation $(e, f, -e, f)$.

Suppose that Claim 1 is not true. Then $e^{*}$ is a twisted loop and $f^{*}$ is proper in $B^{*}$ by Lemma \ref{le-05}. Thus the two ribbon-sides of $e$ lie on a same boundary component of $B$, denoted by $C_{1}$, and if we assign two arrows to the two ribbon-sides of $e$ such that these two arrows are consistent on the edge boundary of $e$, then these two arrows are non-consistent on $C_{1}$  and the two ribbon-sides of $f$ lie on different boundary components of $B$ as in Figure \ref{f0010}.
\begin{figure}[!htbp]
\begin{center}
\includegraphics[width=15cm]{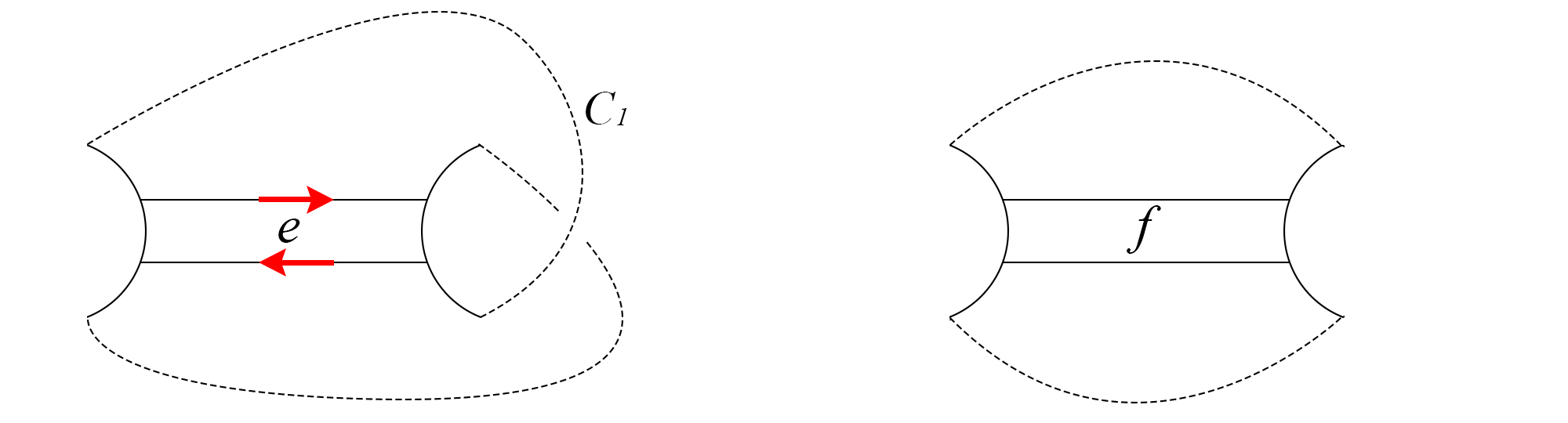}
\caption{Proof of Theorem \ref{le-07}.}
\label{f0010}
\end{center}
\end{figure}
Delete the edge $e$ and note that $f(B)=f(B-e)$ and the two ribbon-sides of $f$ also lie on different boundary components of $B-e$. Hence, $f(B-\{e, f\})=f(B-e)-1$, that is,
$f(B-\{e, f\})=f(B)-1$. Since $\varepsilon(e, f, -e, f)=2$ and
$\varepsilon(B-\{e, f\})=e(B)-1-f(B-\{e, f\})$ by Euler formula,  we have
$$\varepsilon(B^{\{e, f\}})=\varepsilon(e, f, -e, f)+\varepsilon(B-\{e, f\})=e(B)+1-f(B-\{e, f\})$$
by Lemma \ref{le-01}.
Since $\varepsilon(B)=e(B)+1-f(B)$, it is easy to check that  $\varepsilon(B)\neq \varepsilon(B^{\{e, f\}})$,
contrary to ${^{\partial}}\varepsilon_{B}(z)=2^{e(B)}z^{b}$.
The claim then follows.

{\bf Claim 2. } $B$ does not contain a bouquet with signed rotation $(e, f, -e, -f)$.

Assume that Claim 2 is not true. It is easily seen that $B^{e}$ contains a bouquet with signed rotation $(e, f, -e, f)$.
Since ${^{\partial}}\varepsilon_{B^{e}}(z)={^{\partial}}\varepsilon_{B}(z)=2^{e(B)}z^{b}$,  this contradicts Claim 1.

Since $B$ is a non-orientable bouquet, there exists a twisted loop. Let $e$ be any twisted loop.  As $B$ is prime and $e(B)\geqslant 2$, there exists a loop $f$ such that the loops $e$ and $f$ alternate, this contradicts Claim 1 or 2.
Hence $e(B)=1$, that is, $B=B_{\overline{1}}$.
\end{proof}

\subsection{Orientable case}
\noindent

\begin{theorem}\label{le-09}
Let $B$ be a non-empty prime orientable bouquet. Then ${^{\partial}}\varepsilon_{B}(z)=2^{e(B)}z^{b}$ if and only if
$B=B_{2i+1}$.
\end{theorem}
\begin{proof}
The sufficiency is easily verified by Proposition \ref{le-06}. For necessity, the result is easily verified when $e(B)\in \{1, 2\}$.
Assume that $e(B)\geqslant 3$. Let $e, f, g\in E(B)$. Note that $e^{*}, f^{*}$ and $g^{*}$ are proper in $B^{*}$ by  Lemma \ref{le-05}.
Hence the two ribbon-sides of $e$ (or $f$ or $g$) lie on different boundary components of $B$. We denote the two ribbon-sides of $e$ (or $f$ or $g$) lying on the two boundary components of $B$ by $C_{e_{1}}$ and $C_{e_{2}}$ (or $C_{f_{1}}$ and $C_{f_{2}}$ or $C_{g_{1}}$ and $C_{g_{2}}$), respectively.

The following facts about ribbon graphs are well known and readily seen to be true. Deleting any edge $e$ of an orientable ribbon graph $G$ can change the number of boundary components exactly one. Otherwise, $G^{*}$ contains a twisted loop, which is contrary to the orientability of  $G$. More specifically,

  ({\bf T1}) The two ribbon-sides of $e$ lie on different boundary components of $G$ if and only if $f(G-e)=f(G)-1$.

  ({\bf T2}) The two ribbon-sides of $e$ lie on the same boundary component of $G$ if and only if $f(G-e)=f(G)+1$.

From (T1) it follows that $f(B-e)=f(B)-1.$ Obviously, $\varepsilon(B)=e(B)+1-f(B)$ and $\varepsilon(B-\{e, f\})=e(B)-1-f(B-\{e, f\})$ by Euler formula. There are two cases to consider:

{\bf Case 1. } If $B(\{e, f\})=(e, f, e, f)$, we have
$$\varepsilon(B^{\{e, f\}})=\varepsilon(e, f, e, f)+\varepsilon(B-\{e, f\})=e(B)+1-f(B-\{e, f\})$$
by Lemma \ref{le-01}.
Since $\varepsilon(B^{\{e, f\}})=\varepsilon(B)$, it follows that $$f(B-\{e, f\})=f(B)=f(B-e)+1.$$
Then the two ribbon-sides of $f$ must lie on the same boundary component of $B-e$ from (T2). Hence, the two ribbon-sides of $f$ must lie on $C_{e_{1}}$ and $C_{e_{2}}$, respectively, in $B$ . Thus $\{C_{e_{1}}, C_{e_{2}}\}=\{C_{f_{1}}, C_{f_{2}}\}$.

{\bf Case 2. } If $B(\{e, f\})=(e, e, f, f)$, then
$$\varepsilon(B^{\{e, f\}})=\varepsilon(e, e, f, f)+\varepsilon(B-\{e, f\})=e(B)-1-f(B-\{e, f\})$$
by Lemma \ref{le-01}.
As $\varepsilon(B^{\{e, f\}})=\varepsilon(B)$, we have $$f(B-\{e, f\})=f(B)-2=f(B-e)-1.$$ Then the two ribbon-sides of $f$ lie on different boundary components of $B-e$ from (T1). Hence at most one of the two ribbon-sides of $f$ lie on $C_{e_{1}}$ and $C_{e_{2}}$ in $B$. Thus $\{C_{e_{1}}, C_{e_{2}}\}\cap\{C_{f_{1}}, C_{f_{2}}\}\neq \{C_{e_{1}}, C_{e_{2}}\}$.

{\bf Claim 3. } $B$ does not contain a bouquet with signed rotation $(e, f, g, e, g, f)$.

Assume that Claim 3 is not true. Since $B(\{e, f\})=(e, f, e, f)$ and $B(\{e, g\})=(e, g, e, g)$,
it follows that $$\{C_{e_{1}}, C_{e_{2}}\}=\{C_{f_{1}}, C_{f_{2}}\}=\{C_{g_{1}}, C_{g_{2}}\}$$ by Case 1.
Thus $$\{C_{f_{1}}, C_{f_{2}}\}\cap\{C_{g_{1}}, C_{g_{2}}\}=\{C_{f_{1}}, C_{f_{2}}\}.$$ But $B(\{f, g\})=(f, f, g, g)$, this contradicts Case 2.

Suppose that $I(B)$ is not a complete graph.  Note that $I(B)$ is connected. Then there is a vertex set $\{v_{e}, v_{f}, v_{g}\}$ of $I(B)$ such that the induced subgraph $I(B)(\{v_{e}, v_{f}, v_{g}\})$ is a 2-path (see Exercise 2.2.11 \cite{BMU}).
We may assume without loss of generality that the degree of $v_{e}$ is 2 in $I(B)(\{v_{e}, v_{f}, v_{g}\})$ and $v_{e}, v_{f}, v_{g}$ are corresponding to  the loops $e, f, g$ of $B$, respectively. Thus $B(\{e, f, g\})=(e, f, g, e, g, f)$, this contradicts Claim 3. Hence $I(B)$ is a complete graph and $B=B_{2i+1}$ by Proposition \ref{le-06}.
\end{proof}

Theorem \ref{le-09} tells us that Conjecture \ref{con-01} is actually true for all prime orientable bouquets except the family of counterexamples as in Proposition \ref{le-06}.

\section{Concluding remarks}
\noindent

As shown in Remark \ref{rem1}, there are different signed intersection graphs with the same intersection polynomial. More examples could be obtained by using Theorem \ref{main2}. For example, let $K_5^+$ be the positive $K_5$ and $4K_1^-\cup 1K_1^+$ be the disjoint union of 4 negative isolated vertices and 1 positive isolated vertex, then $IP_{K_5^+}(z)=IP_{4K_1^-\cup 1K_1^+}(z)=32z^4$. Similar to the chromatic polynomial \cite{Dong} and the Tutte polynomial \cite{GM}, we could call two signed intersection graphs \emph{IP-equivalent} if they have the same intersection polynomial. It is interesting to find more examples of equivalent signed intersection graphs and eventually clarify the IP-equivalence from the viewpoint of the structures of graphs.

Not every signed graph is a signed intersection graph. We define the intersection polynomial of a signed intersection graph $SG$ to be the partial-dual Euler genus polynomial of the bouquet $B$ with $SG=SI(B)$. Could we redefine the intersection polynomial for signed intersection graphs independent of the bouquets? The recursion in Theorem \ref{le-10} is a try, but fail even for the negative $v_1$. If the answer is negative, could we define a polynomial on a more larger set of signed graphs including all signed intersection graphs such that when we restrict ourself to a signed intersection graph it is exactly the intersection polynomial?

We have characterized non-empty bouquets whose partial-dual genus polynomials have only one term. One can continue to try to characterize bouquets whose partial-dual genus polynomials have exactly two terms.

As we mentioned a little in the introduction, except the partial-dual (i.e. partial-$*$) Euler genus polynomial, there are partial-$\times$, partial-$*\times$, partial-$\times *$ and partial-$*\times*$ Euler genus polynomials \cite{GMT2}. For investigation of the partial-$\bullet$ Euler genus polynomial, one can focus on bouquets if $\bullet\in \{*\times, \times *, *\times *\}$ and focus on quasi-trees (i.e. ribbon graphs with only one face) if $\bullet=\times$. Could we derive something from bouquets or quasi-trees which could determine the partial-$\bullet$ Euler genus polynomial completely?

Finally we point that we find Theorem \ref{le-09} is also obtained independently by Chumutov and Vignes-Tourneret in \cite{chum}, an arXiv paper appeared about one week ago, but the proof is not completely the same.

\section*{Acknowledgements}
\noindent

This work is supported by NSFC (No. 11671336) and the Fundamental Research
Funds for the Central Universities  (No. 20720190062).


\bibliographystyle{model1b-num-names}
\bibliography{<your-bib-database>}

\end{document}